\documentclass[12pt]{article}

\usepackage{amssymb}
\usepackage{amsthm}

\usepackage{caption}
\usepackage{subcaption}
\usepackage{tikz}
\usepackage{amsmath}
\usepackage{multirow}
\usepackage{color}
\newtheorem{Theorem}{Theorem}

\newtheorem{prop}{Proposition}
\newtheorem{mydef}{Definition}[]
\newtheorem{cor}{Corollary}
\newtheorem{rem}{Remark}
\usepackage[affil-it]{authblk} 
\usepackage[round]{natbib}

\begin{document}

%\begin{frontmatter}



\title{The multi-period $p$-center problem with time-dependent travel times}


%\author{Tobia Calogiuri}
%\ead{tobia.calogiuri@unisalento.it}

%\author{Gianpaolo Ghiani}
%\ead{gianpaolo.ghiani@unisalento.it}

%\author{Emanuela Guerriero}%\corref{mycorrespondingauthor}}

%\ead{emanuela.guerriero@unisalento.it}

\author[1]{Tobia Calogiuri} 
 \author[1]{Gianpaolo Ghiani}
 \author[1]{ Emanuela Guerriero\thanks{corresponding author}}
  \author[1]{Emanuele Manni}
\affil[1]{Dipartimento di Ingegneria dell'Innovazione,  Universit\`{a} del Salento, \\Via per Monteroni, 73100 Lecce, Italy }
\providecommand{\keywords}[1]{\textbf{\textit{Keywords---}} #1}
%\ead{emanuele.manni@unisalento.it}

%\address[salento]{Dipartimento di Ingegneria dell'Innovazione,  Universit\`{a} del Salento, \\Via per Monteroni, 73100 Lecce, Italy}
\date{}
\maketitle

\begin{abstract}
This paper deals with a multi-period extension of the $p$-center problem, in which arc traversal times vary over time, and facilities are mobile units that can be relocated multiple times during the planning horizon. \textcolor{black}{The problem arises in several applications, such as emergency services, in which vehicles can be relocated in anticipation of traffic congestion. First, we analyze the problem structure and its relationship with its single-period counterpart, including a special case. Then, the insight gained with this analysis is used to devise a decomposition heuristic.} Computational results on instances based on the Paris (France) road graph indicate that the algorithm is capable of determining good-quality solutions in a reasonable execution time.
\end{abstract}


%\begin{keyword}
%p-center problem \sep time-dependent travel times \sep multi-period location \sep discrete location
%\end{keyword}
%
%\end{frontmatter}


\section{Introduction}\label{intro}
\textcolor{black}{Travel times constitute a key factor when locating service units. The recent availability of detailed traffic data (e.g., Google Traffic) makes it possible to extract historical traffic patterns over various time intervals. 
Such information can be used to adapt location decisions to the varying traffic conditions in order to improve  performance measures. One application that can benefit of such innovations is the location of emergency units, such as ambulances or fire stations \citep{BELANGER20191}}. 
\textcolor{black}{This paper deals with a multi-period location problem, termed \textit{multi-period $p$-center problem with time-dependent travel time} (M$p$CP-TD), in which arc traversal times vary over time, and facilities are mobile units that can be relocated multiple times during the planning horizon. The problem is deterministic in nature since we assume that travel time functions are known as well as the location of users.} \textcolor{black}{ We neglect service dynamics, i.e., we do not take into account the (possible) fleet reduction occurring when some vehicles are responding to demands. Moreover, since users and facilities are located in the same urban area,  we also neglect the relocation dynamics (see, eg. \citep{schmid2010ambulance}).}

\textcolor{black}{ 
In order to highlight the peculiarities of this problem we present an example inspired by the seminal paper of \cite{berman1982locating}. An urban area is served by four emergency vehicles. The service territory is divided into two areas (named A and B) connected by a limited number of streets. 
During off-peak hours, the mobile units are uniformly located (see Figure \ref{fig:sub-first}). During peak hours the streets connecting the two zones become heavily congested. As a result, it is wiser to relocate one unit in zone B and redistribute the remaining units uniformly in zone A (see Figure \ref{fig:sub-second}). It is worth noting that, unlike  \cite{berman1982locating}, in our problem travel times change according to \textit{historical} traffic patterns. Indeed, as stated by \cite{malandraki1992}, variability in travel times has two main components. The first  derives from hourly, daily, weekly and seasonal effects. These traffic factors, modeled by deterministic time-dependent functions, are taken into account in this paper. The
second component of variations which covers random events (such as accidents and weather conditions) is not included in our model.
}
\begin{figure}[t]
%\begin{center}
\begin{subfigure}{.5\textwidth}
  \centering
  \begin{tikzpicture}
[scale=0.60]
\draw[] (0,1) -- (6,1) -- (6,5) -- (0,5) -- (0,1);
\draw (0,0) node[right]{A};
\draw[] (6.5,1) -- (8.5,1) -- (8.5,5) -- (6.5,5) -- (6.5,1);
\draw (6.5,0) node[right]{B};
\draw[->] (5.5,4) -- (7,4);
\draw[<-] (5.5,3.5) -- (7,3.5);
\draw[->] (5.5,2.5) -- (7,2.5);
\draw[<-] (5.5,2) -- (7,2);
\draw (1.5,2) node{$\times$};
\draw (4.5,2) node{$\times$};
\draw (1.5,4) node{$\times$};
\draw (4.5,4) node{$\times$};
\end{tikzpicture}
 \caption{Locations during off-peak hours}
  \label{fig:sub-first}
\end{subfigure} 
\begin{subfigure}{.5\textwidth}
  \centering
\begin{tikzpicture}
[scale=0.60]
\draw[] (0,1) -- (6,1) -- (6,5) -- (0,5) -- (0,1);
\draw (0,0) node[right]{A};
\draw[] (6.5,1) -- (8.5,1) -- (8.5,5) -- (6.5,5) -- (6.5,1);
\draw (6.5,0) node[right]{B};
\draw[->] (5.5,4) -- (7,4);
\draw[<-] (5.5,3.5) -- (7,3.5);
\draw[->] (5.5,2.5) -- (7,2.5);
\draw[<-] (5.5,2) -- (7,2);
\draw (1.5,2) node{$\times$};
\draw (4.5,2) node{$\times$};
\draw (3,4) node{$\times$};
\draw (7.5,3) node{$\times$};
\end{tikzpicture}  
\caption{Locations during peak hours}
  \label{fig:sub-second}
\end{subfigure}
\caption{\textcolor{black}{An example inspired by the seminal paper of \cite{berman1982locating}}}
%\end{center}
\end{figure}

\textcolor{black}{The single-period counterpart of our problem is the} $p$-center problem (\textit{p}CP) which requires to locate a set of $p$ facilities and serve a set of demand sites from the selected locations, with the objective of minimizing the maximum service time between a customer and its assigned facility \citep[see, e.g.,][]{hakimi1964optimum,minieka1970m}. The optimal value is called the \textit{radius}. Contributions in literature usually deal with two main variants: the \textit{vertex} \textit{p}CP, considered here, where the facilities can be located only on the vertices \citep{kariv1979algorithmic}, and the \textit{absolute} \textit{p}CP, in which the facilities can be located on the vertices or on the edges of a graph \citep{CALLAGHAN2017722}. Since the \textit{p}CP is NP-hard \citep{kariv1979algorithmic}, extensive literature exists proposing heuristic methods  based on several paradigms such as tabu search \citep{doi:10.1002/net.10081}, variable neighborhood search \citep{irawan2016hybrid}, and other techniques \citep{tansel2011discrete}. As far as exact \textit{p}CP methods are concerned, the most successful approaches are built on the solution of a series of covering subproblems with the help of reduction and preprocessing techniques \citep[see, e.g.,][]{daskin2000new,chen2009new,calik2013double}. To the best of our knowledge, the most recent exact approach for solving the vertex \textit{p}CP is presented in  \cite{contardo2019scalable}, where the authors  introduce a scalable relaxation-based iterative algorithm. For a comprehensive review, the interested reader can refer to \cite{calik2015p}.

\textcolor{black}{Regarding multi-period location problems, the corresponding planning problems are referred to as \textit{implicit} if all facilities are opened once at the beginning of the planning horizon. In this case, a plan is devised for reallocating demand sites to the selected facilities at specific time instants in response to changes in parameters over time. If facilities are opened and/or closed throughout the planning horizon, then the corresponding planning problems are referred to as \textit{explicit}. In this case a plan is devised for both relocating facilities and reallocating demand sites to the relocated facilities. Recent contributions on multi-period location problems are \cite{Castro2017}, \cite{Escudero2018}, \cite{RAGHAVAN2019507}, and \cite{GUDEN2019615}. In the application context of emergency services, there exists a number of contributions on multi-period \textit{covering} location problems, taking into account the time-dependent variations in travel times and the resulting changes with respect to the corresponding coverage (see, e.g.  \cite{VANDENBERG2015383,schmid2010ambulance,degel2015time,REPEDE1994567,RAJAGOPALAN2008814}}. As stated in \cite{BELANGER20191}, it has become of the highest importance for location analysis to provide approaches seeking equity/fairness as an objective when providing social aid/services. From this point of view,  \textit{p}CP aligns social aid/service to the Rawlsian approach, named after the philosopher John Rawls, which aims to minimize the worst-off served point. In a time-dependent setting, just using a  classical (\textit{single-period}) \textit{p}-center model with maximum (\textit{fixed}) travel times is no longer suitable to model changes of worst-off served points occurring during the planning horizon. \textcolor{black}{ The M$p$CP-TD aims to smooth this loss of geographical equity, trying to keep the centers close to the \textit{time-dependent} worst-off served points during the planning horizon. The planning horizon is divided into several time periods each of which defining specific moments for making adjustments in the system. The goal of the  problem is to locate facilities in such a way that the  sum of the largest service times associated with all the time periods is minimized. } 

To the best of our knowledge, this is the first contribution on multi-period \textit{p}CP with time-dependent travel times (M\textit{p}CP-TD). For a comprehensive review,  the interested reader can refer to  \cite{nickel2015multi}.  

%\textcolor{black}{In this paper, we introduce a variant of the classical $p$-center problem that explicitly incorporates time and traffic conditions in the estimation of the service times. The problem is deterministic in nature since we assume that travel time functions are known as well as the location of users.} The planning horizon is divided into several time periods each of which defining specific moments for making adjustments in the system. 
The remainder of the paper is organized as follows. \textcolor{black}{In Section \ref{formulation} we describe the M$p$CP-TD in detail and formulate it.  Section \ref{property} presents  a property of time-dependent graphs that we call \textit{arc ranking invariance}, which can be exploited to solve the M\textit{p}CP-TD.  In Section \ref{framework}, we present a heuristic solution method and discuss about optimality conditions. Computational tests are reported in Section \ref{results}. The paper ends with an overview of the work done and some conclusions in Section \ref{Conclusions}.}

\textcolor{black}{\section {Problem statement}\label{formulation}
Let $G = (F,C,E)$ be a directed graph, where $F $ and $C $ are the sets of candidate locations of facilities and customer nodes, respectively, and $E = \{(i, j) : i \in F, j \in C\}$ is the set of arcs. %With each arc $(i,j) \in E$ is associated a non-negative travel distance $L_{ij}$.
Let $\tau_{ij}(t)$ be the travel duration of arc $(i, j) \in E$ when demand site $j$  is served from facility $i$ at time $t$ of the planning horizon $[0,T]$. 
We suppose that the traversal times are continuous piecewise linear functions satisfying the \textit{first-in-first-out} (FIFO) property, i.e., leaving the facility $i$ later implies arriving later at demand site $j$. 
According to a multi-period modeling approach, we suppose that $\mathcal{T}$ represents  the set of $M$ time instants during which a relocation of $p$ facilities might be planned, that is:
$$\mathcal{T}=\{t_1,\dots,t_{M-1},t_M\},$$
where  $t_0=0<t_1<\dots<t_{M-1}<t_M<t_{M+1}=T$. \\
In a time-dependent setting, we need to characterize how the worst-case service time varies over time. In particular, we model the time variability of the worst-case service time of each arc as a constant stepwise function, whose \textit{pieces} are the maximum service times between two consecutive possible relocation time instants of $\mathcal{T}$. For this purpose,  we denote with $d_{ij}(\mathcal{I})$  the worst service time of arc $(i,j)\in E$ as:
$$d_{ij}(\mathcal{I})=\max\limits_{t\in\mathcal{I}}\tau_{ij}(t),$$
where $\mathcal{I}$ is a time interval.\\
\indent Given $p$ facilities, a number of $M$ potential time periods, and a value $K \leq M$, the aim of the M$p$CP-TD is to determine a sequence of $K$ \textit{location-allocation} decisions such that the facilities are relocated $K$ times and the sum of the maximum (\textit{worst}) service times associated with each period is minimized.}
%\indent The aim of the M\textit{p}CP-TD is to determine a sequence of \textit{location-allocation} decisions  such that the corresponding number of relocations of $p$ facilities is not greater than a given value $K\leq M$ and the sum of the maximum (\textit{worst}) service times associated with each period is minimized.\\}

\indent  Let $(O_\ell,S_\ell,\mathcal{I}_\ell)$ denote \textit{location-allocation} decisions taken at the beginning of the time period $\mathcal{I}_\ell=[t_\ell,t_{\ell+1}]$, with $t_\ell$ and $t_{\ell+1}$ being two consecutive time instants of $\mathcal{T}$ and $\ell=0,\dots,M$. 
 The \textit{location} component is  the subset   $O_\ell$ modeling the $p$ \textit{open} facilities, that is $O_\ell\subseteq F$ and $|O_\ell|=p$,  with $\ell=0,\dots,M$. The \textit{allocation} component is  encoded as a  vector  $S_\ell$, where the $i$-th element $S_\ell[i]$ is the subset of the demand sites served from the \textit{open} facility $i\in O_\ell$, with  $\ell=0,\dots,M$.
We suppose that, given a time interval $\mathcal{I}_\ell$ and a set of open facilities $O_\ell$, the corresponding vector of allocation decisions  $S_\ell$ is univocally determined by assigning each costumer $j\in C$ to exactly one of the closest open facilities $i\in O_\ell$,  that is:
\begin{equation}\label{seed_loc} 
\textcolor{black}{i\in\arg\min(d_{sj}(\mathcal{I}_\ell)| s\in O_\ell)\Rightarrow j\in S_\ell[i],}
\end{equation}
with $\ell=0,\dots,M$. We synthetically refer to such univocal relationship by asserting that, during time period $\mathcal{I}_\ell$, $O_\ell$ is the $seed$ location decision of $S_\ell$, with $\ell=0,\dots,M$. The maximum service time of $(O_\ell,S_\ell,\mathcal{I}_\ell)$  is denoted with  $r(O,S,\mathcal{I}_\ell)$,  where:
\begin{equation}\label{R_1}
\textcolor{black}{r(O_\ell,S_\ell,\mathcal{I}_\ell)=\max(d_{ij}(\mathcal{I}_\ell)|i \in O_\ell \wedge j \in S_\ell[i]),}
\end{equation}
with $\ell=0,\dots,M$. For the sake of simplicity, from now on, when it is clear we are referring to the time period $\mathcal{I}_\ell$, we denote the corresponding \textit{location-allocation} decisions as $(O_\ell,S_\ell)$, with $\ell=0,\dots,M$.\\

\indent In the M\textit{p}CP-TD, the decision maker is interested in determining a sequence of location-allocation decisions,  one for each time period  $\mathcal{I}_\ell$, with $\ell=0,\dots,M$. We encode such sequence as a pair of vectors $(\mathbf{O}, \mathbf{{S}})$ where:
$$\mathbf{O}=[O_0,\dots,O_M],\quad \mathbf{S}=[S_0,\dots,S_M].$$
The criterion for evaluating each solution $(\mathbf{O}, \mathbf{{S}})$ is the sum of the maximum service times $R(\mathbf{O}, \mathbf{{S}})$, defined as:
$$ R(\mathbf{O}, \mathbf{{S}})=\sum\limits_{\ell=0}^{M} r(O_\ell,S_\ell).$$
We say that  $(\mathbf{O}, \mathbf{{S}})$ prescribes relocations over $\mathcal{T}$, if there exist at least two distinct \textit{seed} location decisions, that is if there exists $\ell$ such that $O_{\ell-1}\neq O_\ell$, with $\ell=1,\dots,M$. Let us denote with $n(\mathbf{O}, \mathbf{{S}})$ the total number of  relocations prescribed by $(\mathbf{O}, \mathbf{{S}})$. We suppose that the decision maker requires that $n(\mathbf{O}, \mathbf{{S}})$ is limited to $K$, with $0\leq K\leq M$.  
Given the former parameters and notation, the M\textit{p}CP-TD can be expressed synthetically as: 
\begin{equation}\label{phi_k_n}
\mathbf{\Phi}(\mathcal{T},K)=\min\limits_{\mathbf{(O,S)}}\left(R(\mathbf{O},\mathbf{S})|n(\mathbf{O},\mathbf{S})\leq K\right).
\end{equation}
\indent In the following sections, we illustrate a decomposition algorithm for solving the optimization problem (\ref{phi_k_n}). The main underlying idea of the proposed heuristic is that each feasible solution of (\ref{phi_k_n}) models a set of \textit{location-allocation} decisions taken according to a two-stage \textit {nested} approach. At the first stage a subset of $K$ time instants $\mathcal{T}_R\subseteq\mathcal{T}$ is selected. At the second stage, a set of $K$ \textit{implicit} multi-period location problems is solved in order to determine one \textit{seed} location decision  for each time instant in $\mathcal{T}_R$. 
%\section{Multi-period formulations}
To ease the description of the proposed heuristic algorithm, we provide an alternative formulation of the M\textit{p}CP-TD as a sequential (\textit{nested}) decision-making process. In the following, such formulation is illustrated by distinguishing two cases referred to as \textit{implicit} multi-period formulation and \textit{explicit} multi-period formulation, respectively.

\subsection{The \textit{implicit} multi-period formulation}
As stated in the literature, the \textit{implicit} multi-period formulation of a location problem implies that no relocations are allowed during the planning horizon, that is $K=0$. In this case the first stage decision can be skipped and  the second stage decision requires the solution of the location problem $\mathbf{\Phi}(\mathcal{T},0)$. 
\begin{mydef}\label{feas_sol_k_0}
A solution $(\mathbf{O}, \mathbf{{S}})$ is feasible  for $\mathbf{\Phi}(\mathcal{T},0)$ if it is characterized by a single  \textit{seed} location decision $O_0$  to be taken at the beginning of the planning horizon, that is $O_\ell=O_0$ with $\ell=1,\dots, M$. 
\end{mydef}
\noindent To ease the discussion, we reformulate $\mathbf{\Phi}(\mathcal{T},0)$ making use of the decision variables of  the classic $p$-center formulation \citep[see,][]{doi:10.1002/9781118032343.ch5} in order to model a solution $(\mathbf{O}, \mathbf{{S}})$. Let $y_i$ be a binary variable modeling the \textit{seed} location decision $i\in O_0$, that is $y_i$ takes value 1 if, during  the planning horizon $[0,T]$, facility $i\in F$ is open and 0 otherwise. The binary variable $x_{ij\ell}$ states whether or not customer $j \in C$ is assigned to facility $i \in F$ during time interval  $\mathcal{I}_\ell$, that is $j\in S_{\ell}[i]$, with $\ell=0,\dots,M$. The continuous variable $r_\ell$ represents the maximum service time (\ref{R_1}), with respect to  period $\mathcal{I}_\ell$, with $\ell=0,\dots,M$. Then, the problem can be formulated as:
\textcolor{black}{\begin{equation} \label{tdcpc_01}
\mathbf{\Phi}(\mathcal{T},0):=\quad\min \sum\limits_{\ell=0}^{M} r_{\ell}
\end{equation} 
s.t.
{\allowdisplaybreaks
\begin{flalign}\label{tdcpc_02}
&\sum\limits_{i\in F}x_{ij\ell}=1  \quad\quad\quad\quad\quad j\in C,\ell=0,\dots,M & \\
\label{tdcpc_03}
&y_{i}\geq x_{ij\ell}  \quad\quad\quad\quad\quad\quad i \in F,j \in C,\ell=0,\dots,M& \\
\label{tdcpc_04}
&\sum\limits_{i\in F}y_{i}= p &  \\
\label{tdcpc_05}
&r_{\ell }\geq \sum\limits_{i\in F}d_{ij}(\mathcal{I}_\ell) x_{ij\ell}  \quad\quad j\in C,\ell=0,\dots,M& \\
\label{tdcpc_06}
&x_{ij\ell}\in\{0,1\}  \quad\quad\quad\quad\quad i\in F,j \in C,\ell=0,\dots,M&\\
\label{tdcpc_07}
&y_{i}\in\{0,1\} \quad\quad\quad\quad\quad\quad i\in F& \\
\label{tdcpc_08} 
&r_{\ell }\geq0  \quad\quad\quad\quad\quad\quad\quad\quad \ell=0,\dots,M.&
\end{flalign}
}}
Objective function (\ref{tdcpc_01}) models the sum of maximum service times $R(\mathbf{O},\mathbf{S})$. Constraints (\ref{tdcpc_02}) ensure that each customer is assigned to one facility. Constraints (\ref{tdcpc_03}) state that a facility is open when at least one customer is allocated to it. Constraint (\ref{tdcpc_04}) states that the total number of facilities  to be opened is $p$. Constraints (\ref{tdcpc_05}) force $r_{\ell}$ to be greater than or equal to the service time from any customer to its assigned facility. Constraints (\ref{tdcpc_06})-(\ref{tdcpc_08}) provide the binary and non-negative conditions on decision variables.
%%%%%%%%%%%%%%%%%
\subsection{The \textit{explicit} multi-period formulation}
If $K>0$, then facilities can be opened and/or closed throughout the planning horizon. As discussed in Section \ref{intro},  in the  literature this case is referred to as \textit{explicit} multi-period modeling approach.
In particular, each feasible solution of $\mathbf{\Phi}(\mathcal{T},K)$ is associated with a subset $\mathcal{T}_R=\{t_{\sigma(1)},\dots,t_{\sigma(K)}\}\subseteq\mathcal{T}$ of $K$ relocation time instants, with $t_0=t_{\sigma(0)}<t_{\sigma(1)}<t_{\sigma(2)}\dots<t_{\sigma(K)}$. 
We observe that $\mathcal{T}_R$ induces a partition of the planning horizon into $K$ macro periods, where  each macro-period $\mathcal{I}^k_R$ starts and ends at two consecutive time instants in $\mathcal{T}_R$. We also observe that $\mathcal{T}_R$  induces a partition  of $\mathcal{T}$ in  $K+1$ subsets $\mathcal{T}^k_R$, each consisting of $M_k$ time instants, with $\sum_{k=0}^{K}M_k=M$. \textcolor{black}{In words, location-allocation decisions at time istant $t_{\sigma(k)}$ remain fixed for all time instants belonging to $\mathcal{T}_R^k$, with $k=0,\dots,K$.} 
We synthetically express the \textit{explicit} multi-period formulation of M\textit{p}CP-TD as:
\begin{equation}\label{phi_k_t}
\mathbf{\Phi}(\mathcal{T},K)=\min\limits_{\mathcal{T}_R}\left(\sum\limits_{k=0}^{|\mathcal{T}_R|}\mathbf{\Phi}(\mathcal{T}^k_R,0)| \mathcal{T}_R\subseteq\mathcal{T}:|\mathcal{T}_R|= K\right),
\end{equation}
where each $\mathbf{\Phi}(\mathcal{T}^k_R,0)$ can be formulated according to model (\ref{tdcpc_01})-(\ref{tdcpc_08}) by substituting the role of the set of time instants $\mathcal{T}$ and the planning horizon $[0,T]$, respectively, with the subset $\mathcal{T}^k_R$ and the macro-period $\mathcal{I}^k_R$, with $k=0,\dots,K$.  We observe that, given a subset $\mathcal{T}_R$, the \textit{inner} optimization problem of (\ref{phi_k_t}) can be decomposed in $K+1$ independent subproblems. Let $(\mathbf{O}(\mathcal{T}_R),\mathbf{S}(\mathcal{T}_R))_k$ denote a feasible solution of  $\mathbf{\Phi}(\mathcal{T}^k_R,0)$, with $k=0,\dots,M$. Quite naturally, we can extend Definition \ref{feas_sol_k_0} as follows. 
\begin{mydef}\label{Def_1}
Let  $(\mathbf{O}(\mathcal{T}_R),\mathbf{S}(\mathcal{T}_R))_k$ denote a feasible solution of $\mathbf{\Phi}(\mathcal{T}^k_R,0)$, stating that one single \textit{seed} location decision $O_{\sigma(k)}$ must be taken at the beginning of the macro-period $\mathcal{I}^k_R$, i.e., $O_\ell=O_{\sigma(k)}$ with $\mathcal{I}_\ell\subseteq\mathcal{I}^k_R$ and $k=0,\dots,K$.
\end{mydef}
\noindent According to (\ref{phi_k_t}) a feasible solution $(\mathbf{O}(\mathcal{T}_R),\mathbf{S}(\mathcal{T}_R))$ of  $\mathbf{\Phi}(\mathcal{T},K)$ can be modeled as a sequence of feasible solutions of the \textit{inner} optimization subproblems, that is:
\begin{equation}\label{alt_O}
(\mathbf{O}(\mathcal{T}_R),\mathbf{S}(\mathcal{T}_R))=[(\mathbf{O}(\mathcal{T}_R),\mathbf{S}(\mathcal{T}_R))_0,\dots,(\mathbf{O}(\mathcal{T}_R),\mathbf{S}(\mathcal{T}_R))_K],
\end{equation}
%In the following sections we investigate the relationship between the arc ranking invariance and the optimality condition of the M\textit{p}CP-TD. For this purpose, a crucial role is played by the special case $\mathbf{\Phi}(\mathcal{T},M)$. 
We  observe that each feasible solution of $\mathbf{\Phi}(\mathcal{T},K)$ is also feasible for $\mathbf{\Phi}(\mathcal{T},M)$, with $0\leq K\leq M$. Indeed, according to formulation (\ref{phi_k_n}),  we have that:
$$n(\mathbf{O},\mathbf{S})\leq K\leq M.$$
This implies that $\mathbf{\Phi}(\mathcal{T},M)$ is a relaxation of $\mathbf{\Phi}(\mathcal{T},K)$.
\begin{rem}\label{Remark_1}
If an optimal  solution  for $\mathbf{\Phi}(\mathcal{T},M)$ prescribes $K$ relocations, then such solution is also optimal for 
 $\mathbf{\Phi}(\mathcal{T},K),\mathbf{\Phi}(\mathcal{T},K+1),\dots,\mathbf{\Phi}(\mathcal{T},M-1)$.
\end{rem}
\noindent Given a reference time interval $\mathcal{I}$, let us denote with $\phi(\mathcal{I})$ the (classical) single-period $p$-center problem defined on $G$ with the service time of  arc $(i,j)\in E$ equal to $d_{ij}(\mathcal{I})$, that is:
$$
\phi(\mathcal{I})=\min\limits_{ (O,S)}r(O,S,\mathcal{I}).
$$
\begin{rem}\label{Remark_2}
When $K=M$ the unique solution of the \textit{outer} optimization problem of (\ref{phi_k_t}) is $\mathcal{T}_R=\mathcal{T}$, with $M_k=1$ and $k=0,\dots,K$. This implies that the optimal solution of $\mathbf{\Phi}(\mathcal{T},M)$ can be determined by solving $M+1$ independent single-period (\textit{classical}) $p$-center problems, that is:
$$
\mathbf{\Phi}(\mathcal{T},M)=\sum\limits_{\ell=0}^{M}\phi(\mathcal{I_\ell}).
$$
\end{rem}
\noindent In the following sections,  we exploit Remarks \ref{Remark_1} and \ref{Remark_2} to devise a set of  sufficient optimality conditions.

\textcolor{black}{\section{The arc ranking invariance property}\label{property}
In this section we investigate a property of  time-dependent graphs that we call \textit{arc ranking invariance}, which can be exploited to solve the M\textit{p}CP-TD.
 As demonstrated in the following sections, when \textit{arc ranking invariance} holds, even if  the \textit{radius} of the graph $G$ is time dependent, the worst-off served demand site does not change during the planning horizon. In this case, there is no need to relocate facilities, i.e., it is suitable an \textit{implicit} multi-period modeling approach \citep{nickel2015multi}. \\
\indent Let us denote with $\mathcal{B}(\mathcal{T})$ the  partition of the planning horizon in $M+1$ time periods $\mathcal{I}_\ell=[t_\ell,t_{\ell+1}]$, with $t_\ell\in\mathcal{T}$ and $\ell=0,\dots,M$. 
\begin{mydef}\label{def_3} \textbf{Arc dominance rule over $\mathcal{B}(\mathcal{T})$.}  Given two arcs $(i,j)$ and $(r,s)$ of $G$ and their travel time functions, $\tau_{ij}(t)$ and $\tau_{rs}(t)$ respectively, we say that $\tau_{ij}(t)$ dominates $\tau_{rs}(t)$ over $\mathcal{B}(\mathcal{T})$ iff for any  time interval $\mathcal{I}_\ell\in\mathcal{B}(\mathcal{T})$:
$$d_{rs}(\mathcal{I}_\ell)\leq d_{ij}(\mathcal{I}_\ell),$$
with $\ell=0,\dots,M$.
\end{mydef}
\begin{mydef}\label{def_4} \textbf{Arc ranking invariance over $\mathcal{B}(\mathcal{T})$.} The time-dependent graph $G$ is ranking invariant over $\mathcal{B}(\mathcal{T})$, if the  arc dominance rule over $\mathcal{B}(\mathcal{T})$ holds for any pair of arcs $(i,j)\in E$ and $(r,s)\in E$.
\end{mydef}
In order to define sufficient conditions for the arc ranking invariance over $\mathcal{B}(\mathcal{T})$, we exploit results provided  in \cite{GRRGHN} and  \cite{cordeau2014analysis}.} In particular, \cite{GRRGHN} proved that any continuous piecewise linear FIFO travel time function can be generated from the model proposed by \cite{ichoua2003vehicle} (IGP model, for short). The authors also proposed an iterative method to determine the IGP parameters of $\tau_{ij}(t)$ on a reference time interval $\mathcal{I}$, that is a constant (dummy length) $L_{ij}$ and a constant stepwise (dummy speed) function $v_{ij}(t)\geq 0$, such that:
\begin{equation}\label{len}
L_{ij}=\int_{t}^{t+\tau_{ij}(t)}v_{ij}(\mu)d\mu,
\end{equation}
where $t\in\mathcal{I}$.  Let us suppose that the IGP parameters have been determined by applying the method proposed in \cite{GRRGHN} with the planning horizon  $[0,T]$ as reference time interval. Without loss of generality, we suppose that all breakpoints of $v_{ij}(t)$ belong to the set of time instants $\mathcal{T}$, that is 
$$v_{ij}(t)=v_{ij\ell},$$
with $t\in \mathcal{I}_{\ell}$, $(i,j)\in E$,  $\ell=0,\dots,M$. Then, we can determine the  factorization of the IGP speeds proposed by \cite{cordeau2014analysis}, that is :
\begin{equation}\label{fac}
v_{ij\ell}=u_{ij} b_{\ell} \delta_{ij\ell},\qquad (i,j) \in E, \ell=0,\dots,M,
\end{equation}
where:
\begin{itemize}
\item $u_{ij}$ is the maximum travel speed across arc $(i,j)\in E$ during the planning horizon $[0,T]$, i.e., $u_{ij}=\max_{\ell=0,\dots,M} v_{ij\ell}$;
\item $b_{\ell} \in [0,1]$ is the best (i.e., lightest) congestion factor during interval $\mathcal{I}_\ell$ on the whole graph, i.e., $b_{\ell}=\max_{(i,j)\in A} v_{ij\ell} / u_{ij}$;
\item $\delta_{ij\ell} = \frac{v_{ij\ell} / u_{ij}}{ b_{\ell}}$ varies in $[0,1]$ and represents the degradation of the congestion factor of arc $(i,j)$ in interval $\mathcal{I}_{\ell}$ with respect to the less congested arc.
\end{itemize}
Let us denote with $\Delta=\min_{i,j,\ell}\delta_{ij\ell}$ the heaviest degradation of the congestion factor of any arc $(i,j)\in E$ during the planning horizon. The following theorem states a sufficient condition for the arc ranking invariance property over  $\mathcal{B}(\mathcal{T})$.
%In a time-dependent setting, we need to characterize how the worst-case service time varies over time. In particular, we model the time variability of the worst-case service time of each arc as a constant stepwise function, whose \textit{pieces} are the maximum service times between two consecutive possible relocation time instants of $\mathcal{T}$. For this purpose, we introduce the following definitions and notation.
%Given a reference time period $\mathcal{I}$, we determine the worst service time $d_{ij}(\mathcal{I})$ of arc $(i,j)\in E$ as:
%$$d_{ij}(\mathcal{I})=\max\limits_{t\in\mathcal{I}}\tau_{ij}(t).$$
%\begin{mydef}\label{def_3} \textbf{Arc dominance rule over $\mathcal{B}(\mathcal{T})$.}  Given two arcs $(i,j)$ and $(r,s)$ of $G$ and their travel time functions, $\tau_{ij}(t)$ and $\tau_{rs}(t)$ respectively, we say that $\tau_{ij}(t)$ dominates $\tau_{rs}(t)$ over $\mathcal{B}(\mathcal{T})$ iff for any  time interval $\mathcal{I}_\ell\in\mathcal{B}(\mathcal{T})$:
%$$d_{rs}(\mathcal{I}_\ell)\leq d_{ij}(\mathcal{I}_\ell),$$
%with $\ell=0,\dots,M$.
%\end{mydef}
%\begin{mydef}\label{def_4} \textbf{Arc ranking invariance over $\mathcal{B}(\mathcal{T})$.} The time-dependent graph $G$ is ranking invariant over $\mathcal{B}(\mathcal{T})$, if the  arc dominance rule over $\mathcal{B}(\mathcal{T})$ holds true for any pair of arcs $(i,j)\in E$ and $(r,s)\in E$.
%\end{mydef}
\begin{Theorem}\label{delta_1}
Given a time-dependent graph $G$,  if  $\Delta=1$, then the arc ranking invariance property holds  over  $\mathcal{B}(\mathcal{T})$.
\end{Theorem}
\begin{proof}
Since $\Delta=1$, then in speed variation law (\ref{fac}) we have that $\delta_{ij\ell}=1$. This implies that for any given pair  of arcs $(i,j)\in E$ and $(r,s)\in E$ and for any start time $t\in[0,T]$, the relationship (\ref{len}) can be rewritten as:
$$
\frac{L_{ij}}{u_{ij}}=\int_{t}^{t+\tau_{ij}(t)}b(\mu)d\mu,
$$

$$
\frac{L_{rs}}{u_{rs}}=\int_{t}^{t+\tau_{rs}(t)}b(\mu)d\mu,
$$
where $b(t)=b_\ell$ with $t\in\mathcal{I}_{\ell}$, $\ell=0,\dots,M$.  We observe that:
\begin{equation}\label{order_1}
\frac{L_{ij}}{u_{ij}}=\int_{t}^{t+\tau_{ij}(t)}b(\mu)d\mu\leq\int_{t}^{t+\tau_{rs}(t)}b(\mu)d\mu=\frac{L_{rs}}{u_{rs}}\Leftrightarrow\tau_{ij}(t)\leq\tau_{rs}(t).
\end{equation}
Since (\ref{order_1}) holds for any start travel time $t\in [0,T]$, then the thesis is proved.
\end{proof}

\noindent In a typical time-dependent setting, the facility \textit{location/allocation} decisions must take into account that the underlying graph $G$ might satisfy the hypothesis of Theorem \ref{delta_1} only during portions of the planning horizon. In order to overcome this issue we allow facilities to be relocated $K$ times throughout
the planning horizon, in an anticipatory manner. Therefore, given a set of $K$ time instants $\mathcal{T}_R$, we can extend, in a quite natural way, the notation of  \cite{ichoua2003vehicle}  to the  $K+1$ macro periods $\mathcal{I}_R^k$, with $k=0,\dots,K$. In the following,  we suppose to apply the iterative method proposed in \cite{GRRGHN} to each arc $(i,j)\in E$, by considering the macro period $\mathcal{I}^k_R$ as reference time interval, with $k=0,\dots, K$. In this way, for each arc $(i,j)\in E$, we obtain $K+1$ distinct speed factorizations (\ref{fac}), one for each macro-period. Then, we compute $\Delta_k(\mathcal{T}_R)$, that is the heaviest degradation of the congestion factor of any arc $(i,j)\in E$ during macro-period $\mathcal{I}^k_R$, with $k=0,\dots, K$. Finally, we compute $\Delta(\mathcal{T}_R)=\min_k\Delta_k(\mathcal{T}_R)$. Corollary \ref{delta_2} states a sufficient condition for arc ranking invariance. For the sake of notational convenience, as we did for $\mathcal{T}$, we denote with  $\mathcal{B}(\mathcal{T}^k_R)$ the subset of time-periods in $\mathcal{B}(\mathcal{T})$ belonging to the macro-period $\mathcal{I}^k_R$, with $k=0,\dots, K$. \textcolor{black}{ In words, $\mathcal{B}(\mathcal{T}_R^k)$ is the subset of time periods during which location-allocation decisions at time instant $t_{\sigma(k)}$ remain fixed, with $k=0,\dots,K$.}
\begin{cor}\label{delta_2}
Given a time-dependent graph $G$,   and a subset of time instants $\mathcal{T}_R$, if  $\Delta(\mathcal{T}_R)=1$, then the arc ranking invariance property holds over  each $\mathcal{B}(\mathcal{T}^k_R)$, with $k=0,\dots K$.
\end{cor}
\begin{proof}
Since each $\Delta_k(\mathcal{T}_R)\in [0,1]$, by the hypothesis we have that $\Delta_k(\mathcal{T}_R)=1$ for all $k=0,\dots,K$. Then, from Theorem \ref{delta_1}, the thesis is proved.
\end{proof}

\noindent In the following, we propose a heuristic solution method that aims to determine a subset $\mathcal{T}_R$ such that $\Delta(\mathcal{T}_R)\simeq1.$ \textcolor{black}{ We also prove that when $\Delta(\mathcal{T}_R)=1$, then the proposed solution approach is optimal. }

\section{Solution approach}\label{framework}
According to (\ref{alt_O}), in order to outline a heuristic algorithm, we need to devise  two hierarchical \textit{nested} phases. During the first phase the algorithm determines a subset $\mathcal{T}_R$. During the second phase, we determine a sequence of $K+1$ \textit{seed} location decisions $[O_{\sigma(0)},\dots,O_{\sigma(K)}]$, one for each time instant included in $\mathcal{T}_R$. The two phases are described in the following. \\ \\
\textbf{Phase I.} \textit{Choose a subset of time instant $\mathcal{T}_R\subseteq \mathcal{T}$}. This step aims to determine a solution for the \textit{outer} optimization problem (\ref{phi_k_t}), which could be modeled  as a set partitioning problem on $\mathcal{T}$. The main issue is how to determine a cost partitioning function approximating the sum of the optimal values  $\mathbf{\Phi}(\mathcal{T}^k_R,0)$, with $k=0,\dots,K$. To overcome this drawback, we exploit  a set of optimality conditions stating that if  a subset $T_R$ satisfies the hypothesis of Corollary 1, then the second phase of our heuristic method determines the optimal solution.  For this reason, during Phase I, we solve  \textcolor{black}{a  binary program} aiming to determine the subset  $\mathcal{T}_R$ maximizing a proxy value of $\Delta(\mathcal{T}_R)$. Finally, we observe that this phase can be skipped for the special cases $\mathbf{\Phi}(\mathcal{T},M)$ and $\mathbf{\Phi}(\mathcal{T},0)$, where there exists a single solution to the \textit{outer} optimization problem of (\ref{phi_k_t}), that is, $\mathcal{T}_R=\mathcal{T}$ and $\mathcal{T}_R=\emptyset$, respectively.
\\ \\
\textbf{Phase II.} \textit{Choose one seed selection decision for each $\mathbf{\Phi}(\mathcal{T}^k_R,0)$, with $k=0,\dots,K$.} Given the subset  $\mathcal{T}_R$  selected during Phase I, this step aims to determine one feasible solution $(\bar{\mathbf{O}}(\mathcal{T}_R),\bar{\mathbf{S}}(\mathcal{T}_R))_k$ for each \textit{inner} optimization independent sub-problem of (\ref{phi_k_t}).  Let $\bar{O}_k$  denote the location decision prescribed by the optimal solution  of the \textit{classical} p-center problem $\phi(\mathcal{I}^k_R)$, with $k=0,\dots,K$. The \textit{seed} location decision of  $(\bar{\mathbf{O}}(\mathcal{T}_R),\bar{\mathbf{S}}(\mathcal{T}_R))_k$ is set equal to $\bar{O}_k$.
\\ \\
Finally, the set of $K+1$ solutions determined during  Phase II are converted in a complete solution  $(\bar{\mathbf{O}}(\mathcal{T}_R),\bar{\mathbf{S}}(\mathcal{T}_R))$ according to (\ref{alt_O}).

\subsection{Linking arc ranking invariance and optimality}
If during Phase I, the subset $\mathcal{T}_R$ is  empty, then  $(\bar{\mathbf{O}}(\emptyset),\bar{\mathbf{S}}(\emptyset))$  prescribes no relocation, that is  $n(\bar{\mathbf{O}}(\emptyset),\bar{\mathbf{S}}(\emptyset))=0$. In the following, we prove a sufficient optimality condition for this special case. 

\begin{prop}\label{Prop_AG0}
Given a set of time instants $\mathcal{T}$, if the time-dependent graph $G$ is ranking invariant over $\mathcal{B}(\mathcal{T})$, then any feasible solution of $\mathbf{\Phi}(\mathcal{T},0)$ prescribes the same \textit{location-allocation} decision for each time period, that is:
$$
(O_\ell,S_\ell)=(O_0,S_0),
$$
with $\ell=1,\dots,M$. 
\end{prop}
\begin{proof}
From the hypothesis on the ranking invariance over  $\mathcal{B}(\mathcal{T})$ and the \textit{seed} location definition  (\ref{seed_loc}), it descends that, given two distinct time periods $\mathcal{I}_\ell$ and $\mathcal{I}_\ell^\prime$:
$$
O_\ell=O_{\ell^\prime}\Rightarrow S_\ell=S_{\ell^\prime},
$$
with $\ell\neq\ell^\prime$ and  $\ell,\ell^\prime=0,\dots,M$.  According to (\ref{K_0}), each feasible solution of  $\mathbf{\Phi}(\mathcal{T},0)$ can be denoted as $(\mathbf{O}(\emptyset),\mathbf{S}(\emptyset))$. From Definition \ref{feas_sol_k_0} it follows that a feasible solution of  $\mathbf{\Phi}(\mathcal{T},0)$ is characterized by a unique \textit{seed} location decision, that is:
$$
(\mathbf{O}(\emptyset),\mathbf{S}(\emptyset))\Leftrightarrow O_\ell=O_0\quad\quad \ell=1,\dots,M.
$$
Then, the thesis is proved.
\end{proof}
Let us denote with $\hat{O}_\ell$ the location decision prescribed by the optimal solution of the single-period $p$-center problem $\phi(\mathcal{I_\ell})$, with $\ell=0,\dots,M$. 

\begin{prop}\label{Prop_AG1}
Given a set of time instants $\mathcal{T}$, if the time-dependent graph $G$ is ranking invariant over $\mathcal{B}(\mathcal{T})$, then the \textit{location-allocation}  decision $(\hat{O}_0,\hat{S}_0)$ is also optimal for any single-period (\textit{classical}) $p$-center problem $\phi(\mathcal{I}_\ell)$, with $\ell=1,\dots,M$.
\end{prop}
\begin{proof}
From the hypothesis on the ranking invariance over  $\mathcal{B}(\mathcal{T})$ and Proposition \ref{Prop_AG0}, it descends that the worst service time of a generic single-period decision $(O,S)$ is associated with the same arc for every pair of time periods $\mathcal{I}_\ell$ and $\mathcal{I}_{\ell'}$, that is:
$$\arg\max\limits_{(i,j)\in E}(d_{ij}(\mathcal{I}_\ell) ~|~ i \in O,  j \in S[i] )=$$$$=\arg\max\limits_{(i,j)\in E}(d_{ij}(\mathcal{I}_{\ell'})~ |~ i \in O,  j \in S[i]),$$
where $\ell\neq \ell'$ and $\ell,\ell'=0,\dots,M$.
This means that the arc ranking invariance implies a worst service time ranking invariance that is:
$$
r(\hat{O}_0,\hat{S}_0,\mathcal{I}_\ell)\leq r(O,S,\mathcal{I}_\ell)\Leftrightarrow r(\hat{O}_0,\hat{S}_0,\mathcal{I}_{\ell'})\leq r(O,S,\mathcal{I}_{\ell'}),
$$
where $\ell\neq \ell'$, $(\hat{O}_0,\hat{S}_0)\neq(O,S)$ and $\ell,\ell'=0,\dots,M$. Since a single-period $p$-center problem aims to determine a solution with the minimum worst service time, then the thesis is proved. 
\end{proof}
Under the same hypothesis of Proposition \ref{Prop_AG1}  it is possible to demonstrate also the optimality of $(\mathbf{\bar{O}}(\emptyset),\mathbf{\bar{S}}(\emptyset))$.
\begin{prop}\label{Prop_AG2}
Given a set of time instants $\mathcal{T}$, if the time-dependent graph $G$ is ranking invariant over $\mathcal{B}(\mathcal{T})$, then the solution $(\mathbf{\bar{O}}(\emptyset),\mathbf{\bar{S}}(\emptyset))$ is optimal for any
$\mathbf{\Phi}(\mathcal{T},K)$, with $k=0,\dots,M$.
\end{prop}
\begin{proof}
 \textcolor{black}{From Remark \ref{Remark_1} it follows that the thesis is proved if we demonstrate that $(\mathbf{\bar{O}}(\emptyset),\mathbf{\bar{S}}(\emptyset))$ is optimal for $\mathbf{\Phi}(\mathcal{T},M)$. 
 Let us denote with $(\mathbf{\hat{O}}(\emptyset),\mathbf{\hat{S}}(\emptyset))$ a feasible solution of $\mathbf{\Phi}(\mathcal{T},0)$ having as (unique) \textit{seed} location $\hat{O}_0$.
 From Remark \ref{Remark_2} and Proposition \ref{Prop_AG1}, it results that $(\mathbf{\hat{O}}(\emptyset),\mathbf{\hat{S}}(\emptyset))$ is optimal for  $\mathbf{\Phi}(\mathcal{T},M)$. Therefore,}
 the thesis is proved if we demonstrate that:
 \begin{equation}\label{Prop_AG2:0}
(\bar{O}_0,\bar{S}_0)=(\hat{O}_0,\hat{S}_0).
 \end{equation}
We prove (\ref{Prop_AG2:0}) by contradiction. Therefore, by assuming that $(\bar{O}_0,\bar{S}_0)\neq(\hat{O}_0,\hat{S}_0)$, we have that  $(\hat{O}_0,\hat{S}_0)$ is not optimal for $\phi([0,T])$, that is:
  \begin{equation}\label{Prop_AG2:1}
r(\bar{O}_0,\bar{S}_0,[0,T])<r(\hat{O}_0,\hat{S}_0,[0,T]).
 \end{equation}
 Let $\ell'$ denote the time interval in $\mathcal{B}(\mathcal{T})$ associated to the worst service time of $(\hat{O}_0,\hat{S}_0)$, that is:
 $$r(\hat{O}_0,\hat{S}_0,[0,T])=r(\hat{O}_0,\hat{S}_0,\mathcal{I}_{\ell'})$$
 Since $\mathcal{I}_{\ell'}\subseteq [0,T]$, then we have that:
  \begin{equation}\label{Prop_AG2:2}
r(\bar{O}_0,\bar{S}_0,\mathcal{I}_{\ell'})\leq r(\bar{O}_0,\bar{S}_0,[0,T]).
 \end{equation}
From (\ref{Prop_AG2:1}) and (\ref{Prop_AG2:2}), it follows that:
$$
r(\bar{O}_0,\bar{S}_0,\mathcal{I}_{\ell'})<r(\hat{O}_0,\hat{S}_0,\mathcal{I}_{\ell'}),
$$
which implies that $(\hat{O}_0,\hat{S}_0)$ is not optimal for $\phi(\mathcal{I}_{\ell'})$. This contradicts the thesis of Proposition  \ref{Prop_AG1} and, therefore, the corresponding hypothesis, which are the same we are making.
\end{proof}

\noindent Given a subset $\mathcal{T}_R\neq\emptyset$,  from Propositions  \ref{Prop_AG1} and  \ref{Prop_AG2} it follows that, if the arc ranking invariance property holds over $\mathcal{B}(\mathcal{T}^k_R)$, then the subsequence $(\mathbf{\bar{O}}(\mathcal{T}_R),\mathbf{\bar{S}}(\mathcal{T}_R))_k$ is optimal for $\mathbf{\Phi}(\mathcal{T}^k_R,0)$, with $k=0,\dots, K$. From (\ref{phi_k_t}) it descends that we can generalize the sufficient optimality condition to the complete solution $(\mathbf{\bar{O}}(\mathcal{T}_R), \mathbf{\bar{S}}(\mathcal{T}_R))$ as follows.
\begin{prop}\label{Special_case}
Given a subset of $K$ time instants $\mathcal{T}_R\subset\mathcal{T}$, if  the arc ranking invariance property holds over each $\mathcal{B}(\mathcal{T}^k_R)$, with $k=0,\dots, K$, then 
$(\mathbf{\bar{O}}(\mathcal{T}_R), \mathbf{\bar{S}}(\mathcal{T}_R))$ is optimal for all location problems 
$\mathbf{\Phi}(\mathcal{T},K),\dots,\mathbf{\Phi}(\mathcal{T},M)$.
\end{prop}
\begin{proof}
Propositions (\ref{Prop_AG1}) and (\ref{Prop_AG2}) can be generalized by asserting that $(\hat{O}_{\sigma(k)},$ $\hat{S}_{\sigma(k)})$ and$(\bar{O}_{\sigma(k)},\bar{S}_{\sigma(k)})$ are both optimal for each  single-period $p$-center problem $\phi(\mathcal{I}_\ell)$, where $\mathcal{I}_\ell\in\mathcal{B}(\mathcal{T}^k_R)$ and $k=0,\dots,K$. Therefore, according to Remark \ref{Remark_2}, it descends that $(\mathbf{\bar{O}}(\mathcal{T}_R),\mathbf{\bar{S}}(\mathcal{T}_R))$ is an optimal solution of $\mathbf{\Phi}(\mathcal{T},M)$. From Remark \ref{Remark_1}, the thesis is proved.
\end{proof}

\noindent From Theorem \ref{delta_1} and  Proposition \ref{Special_case}, it descends the following Corollary. 
\begin{cor}\label{Special_casexx}
Given a  subset of $K$ time instants $\mathcal{T}_R\subset\mathcal{T}$, if  $\Delta(\mathcal{T}_R)=1$, then 
$(\mathbf{\bar{O}}(\mathcal{T}_R), \mathbf{\bar{S}}(\mathcal{T}_R))$ is optimal for 
$\mathbf{\Phi}(\mathcal{T},K)$.
\end{cor}
\noindent In the following, we propose an optimization model that aims to determine a subset $\mathcal{T}_R$ such that $\Delta(\mathcal{T}_R)\simeq1.$
 
\subsection{Selection of relocation time instants}
Given a subset $\mathcal{T}_R\subseteq\mathcal{T}$, evaluating the corresponding $\Delta(\mathcal{T}_R)$ requires determining each  $\Delta_k(\mathcal{T}_R)$ according to the following two steps, with $k=0,\dots,K$. First, we should run the procedure proposed in \cite{GRRGHN} for each arc $(i,j)\in E$ taking as  reference time interval the macro period \textcolor{black}{$\mathcal{I}^k_R$ with $k=0,\dots,K$}. Then,  the speed decomposition (\ref{fac}) and the corresponding $\Delta_k(\mathcal{T}_R)$ should be determined.

We approximate such computing procedure according to the following two-steps procedure.
\\\\
\underline{\emph{STEP 1.}} We run the procedure proposed in \cite{GRRGHN} for each arc $(i,j)\in E$ taking as  reference time period the overall planning horizon $[0,T]$. To ease the discussion, we rewrite the corresponding speed decomposition (\ref{fac}) with overlined symbols, that is:
$$\bar{v}_{ij\ell}=\bar{b}_\ell\bar{\delta}_{ij\ell}\bar{u}_ij,$$
with $(i,j)\in E$ and $\ell=0,\dots,M$.
\\\\
\underline{\emph{STEP 2.}} We evaluate the given subset $\mathcal{T_R}=\{t_{\sigma(1)},\dots,t_{\sigma(K)}\}\subseteq\mathcal{T}$ according to the parameter $z(\mathcal{T}_R)\leq1$, which is  a proxy value of  $\Delta(\mathcal{T}_R)$ defined as follows:
$$
z(\mathcal{T}_R)=\min\limits_{k=0,\dots,K}c_{\sigma(k)\sigma(k+1)},
$$
with
\begin{equation}\label{c_coeff}
c_{\sigma(k)\sigma(k+1)}=\min(\bar{\delta}_{ij\ell}| \mathcal{I}_\ell\subseteq[t_{\sigma(k)},t_{\sigma(k+1)}]\wedge (i,j)\in E),
\end{equation}
where we recall that the interval $[t_{\sigma(k)},t_{\sigma(k+1)}]$ is the $k$-th macro period $\mathcal{I}^k_R$, with $k=0,\dots,K$. 
\\\\
The main issue of such approach is that for any subset of time instants $\mathcal{T}_R$, it results that:
$$z(\mathcal{T}_R)=\Delta(\emptyset).$$ 
Nevertheless, we observe that the optimality condition on $\Delta(\mathcal{T}_R)$ can be reformulated as follows:
$$\Delta(\mathcal{T}_R)=1\Longleftrightarrow\sum\limits_{k=0}^{K}\Delta(\mathcal{T}^k_R)=K+1.$$
Therefore, during Phase I, we select the subset $\mathcal{T}_R$ having the maximum value of 
\begin{equation}\label{aa}
\sum\limits_{k=0}^{K}z(\mathcal{T}^k_R).
\end{equation}
We model such selection decision as the  \emph{Resource Constrained Maximum Path Problem} (\ref{1c1})-(\ref{1c8}).

Let $\overline{G}=(V,A)$ denote an acyclic directed graph, where  $V=\{0,1,\dots,M\}$ is the set of nodes and  $A$ denote the set of $\binom{M}{2}$ pairs of nodes  $(h,\ell)$, 
$$A=\{(h,\ell)~|~h,\ell\in V\wedge h<\ell\}.$$ 
In the proposed model,  a subset $\mathcal{T}_R$ is modeled as the simple path $\{0,\sigma(1),$ $\dots,\sigma(K),M\}$, i.e., a path starting at node $0$, ending at node $M$ and consisting of $K+1$ arcs. 
In particular, for each arc $(h,\ell)\in A$ we compute the \textit{gain} coefficient $c_{h\ell}$ according to (\ref{c_coeff}) and define the binary decision variable $y_{h\ell}$, taking value 1 if  the arc $(h,\ell)$ belongs to the path. 

\begin{align}
& \mathrm{Maximize}  & &  \sum\limits_{(h,\ell)\in A} c_{h\ell}y_{h\ell} & \label{1c1}\\
& \mathrm{s.t.}  & &   & \nonumber \\
& & & \sum\limits_{(h ,\ell)\in A} y_{h\ell}- \sum\limits_{(\ell,h)\in A}y_{\ell h}=0 &\ell \in \{1,\dots,M\} \label{2c6}\\
& & & \sum\limits_{(0,\ell)\in A} y_{0\ell} = 1 & \label{1c4}\\
& & & \sum\limits_{(h,M)\in A} y_{h M} = 1 & \label{1c5}\\
& & & \sum\limits_{(h,\ell)\in A} y_{h\ell} = K+1 & \label{1c7}\\
& & & y_{h\ell}\in\{0,1\}& (h,\ell)\in A & \label{1c8}
\end{align}
The objective function (\ref{1c1}) models (\ref{aa}) as the cost of a path on $\overline{G}$. Constraints (\ref{2c6})--(\ref{1c5}) are flow conservation constraints, while constraint (\ref{1c7}) is the resource constraint stating that a feasible path consists of $K+1$ arcs. Constraints (\ref{1c8}) provide the binary condition on decision variables.

\section{Computational results}\label{results}
The proposed heuristic algorithm was coded in Java and run on a MacBook Pro with an  Intel Core 2 Duo processor clocked at 2.33 GHz and 4 GB of memory. The $K+1$ time-invariant $p$-center problems  $\phi(\mathcal{I}^k_R)$ of Phase II were modeled as in  \cite{doi:10.1002/9781118032343.ch5}. The optimization models of Phases I and II were solved by Cplex 12.6.0 \textcolor{black}{with a time limit of 900 seconds}.
Our approach was tested on a set of instances derived from the road network of the urban area of Paris (France) covering $2531.4$ $\mathrm{km}^2$. Spatial data were extracted from OpenStreetMap (www.openstreetmap.org). The road-network graph consisted of  $307,998$ arcs: $151,353$ streets were characterized by time-dependent travel times, whilst the remaining  $156,645$ streets had constant travel times. Using such realistic traffic data, we generated 240 instances. Following the literature, in all the test sets each vertex was both a customer and a candidate location for a facility (i.e., $C\equiv F$). 
For each possible value of $|C|$ in the set $\{50,100\}$, we  generated 10 time-dependent graphs. As far as the multi-period setting was concerned,  the  set $\mathcal{T}$ consisted of 120 time instants. For each time-dependent graph, we generated 12 M$p$CP-TD instances, one for each possible pair $(p,K)$, with $p\in\{5,10\}$ and $K\in\{0,2,4,6,8,10\}$.  

In Tables \ref{Tab1} and \ref{Tab2}  we report our computational results for $|C|=50$ and $|C|=100$, respectively. In both tables, the first two columns are self-explanatory. The remaining column headings are as follows:
\begin{itemize}
\item $GAP$: the percentage optimality gap;
\item PHASE I: time (in seconds) spent to determine the subset $\mathcal{T}_R$;
\item PHASE II: time (in seconds) spent to determine the $K+1$ \textit{seed} decisions, one for each macro-period $\mathcal{I}^k_R$, with $k=0,\dots,K$;
\item TIME: overall time (in seconds) spent by the proposed algorithm to determine the heuristic solution  $(\bar{\mathbf{O}}(\mathcal{T}_R),\bar{\mathbf{S}}(\mathcal{T}_R))$.
\item \textcolor{black} {GAIN:  improvement of the objective function value with respect to $\mathbf{\Phi}(\mathcal{T},0)$. The value has been normalized with respect to the maximum value, i.e. :
$$\frac{\mathbf{\Phi}(\mathcal{T},0)- R(\bar{\mathbf{O}}(\mathcal{T}_R),\bar{\mathbf{S}}(\mathcal{T}_R))}{\mathbf{\Phi}(\mathcal{T},0)-\mathbf{\Phi}(\mathcal{T},M)}.$$}
\end{itemize}
\textcolor{black} {The columns reporting the optimality gap, the execution times and  gains are averaged across all instances.}

The average gap value was $3.96\%$ and $5.44\%$ for $|C|=50$ and $|C|=100$, respectively. 

If $K=M$, facility relocations are allowed at each variation of the worst case service time. As stated in the previous sections, in this case  the optimal solution can be determined by solving $M$ independent \textit{single-period} $p$-center problems. Moreover, $\mathbf{\Phi}(\mathcal{T},M)$ represents a lower bound for the M$p$CP-TD with $0\leq K<M$. For these reasons, the optimality gaps are computed as:
$$\frac{R(\bar{\mathbf{O}}(\mathcal{T}_R),\bar{\mathbf{S}}(\mathcal{T}_R))-\mathbf{\Phi}(\mathcal{T},M)}{\mathbf{\Phi}(\mathcal{T},M)}.$$

\begin{table}[!t]
\centering
\caption{Computational results of the instances with $|C|=50$}
\label{Tab1}
\begin{small}
\begin{tabular}{c|c|c|c|c|c|c}
\hline
$p$                   & $K$   & $GAP$ & PHASE I& PHASE II& TIME & \textcolor{black}{GAIN}\\ \hline
\multirow{7}{*}{5} &0 & 8.57& 0.00  &0.92& 0.92& \textcolor{black}{0.0  }  \\ %\cline{2-6} 
                     & 2   & 4.70   &14.59  &  1.85& 16.44 & \textcolor{black}{0.38}\\ %\cline{2-6} 
                     & 4   & 4.18   & 35.60&4.62  & 40.22 &\textcolor{black}{ 0.45}\\ %\cline{2-6} 
                     & 6   & 3.41   &112.33  & 6.48& 118.81& \textcolor{black}{0.55}\\ %\cline{2-6} 
                     & 8  &3.30   & 170.91  & 8.33 &179.24& \textcolor{black}{0.58}\\ %\cline{2-6} 
                     & 10  & 2.65  & 237.79& 10.18 & 247.97&\textcolor{black}{0.65}\\ %\cline{2-6} 
                     & 120 & 0.00   &0.00 & 111.20& 111.20 &\textcolor{black}{ 1.0}  \\ \hline
\multirow{7}{*}{10} &0& 8.64 &0.00  &0.74& 0.74&  \textcolor{black}{0.0 } \\ %\cline{2-6} 
                     & 2   & 5.72   &13.44& 2.23& 15.67& \textcolor{black}{ 0.27}  \\ %\cline{2-6} 
                     & 4   & 4.36   &35.66&3.71& 39.32 & \textcolor{black}{0.42 } \\ %\cline{2-6} 
                     & 6   & 3.25  &112.70&5.20& 117.90&\textcolor{black}{ 0.60 }  \\ %\cline{2-6}  
                     & 8  & 3.35  & 170.60&6.70& 177.26 &\textcolor{black}{ 0.60}\\ %\cline{2-6}  	
                     & 10  & 3.30  &238.49&8.17& 246.66& \textcolor{black}{0.60}   \\ %\cline{2-6} 
                     & 120 & 0.00  &0.00 & 89.11& 89.11  & \textcolor{black}{1.0} \\ \hline
\multicolumn{2}{l|}{Average} & 3.96 & 81.58 & 18.53 & 100.11&\textcolor{black}{0.49} \\ \hline
\end{tabular}
\end{small}
\end{table}

\noindent If the number of relocations $K=0$, the M$p$CP-TD is solved as a \textit{single-period} $p$-center problem. As demonstrated in the previous sections, such an approach is optimal if arc ranking invariance holds. In our instances, such an optimality condition was never satisfied, resulting in an average gap equal to $8.60\%$ and $11.30\%$ for Tables \ref{Tab1} and \ref{Tab2}, respectively.
 
For $K=6$,  the average gap was consistently lower than $3.50\%$ in Table \ref{Tab1} and $5.50\%$ in Table \ref{Tab2}. For $K=8$ and $K=10$, slight improvements (lower than $1\%$ in both tables) were observed.

\begin{table}[!t]
\centering
\caption{Computational results of the instances with $|C|=100$}
\label{Tab2}
\begin{small}
\begin{tabular}{c|c|c|c|c|c|c}
\hline
$p$                   & $K$   & $GAP $ & PHASE I& PHASE II& TIME& \textcolor{black}{GAIN} \\ \hline
\multirow{7}{*}{5}  &0& 10.39 & 0.00 & 7.75&7.75 & \textcolor{black}{0.00}  \\ 
                     & 2   & 7.62   &81.31&23.25& 104.56  & \textcolor{black}{0.19}  \\ 
                     & 4   & 5.51   &104.76&38.75&143.51 &  \textcolor{black}{0.41} \\ 
                     & 6   & 4.50   &145.47&54.26& 199.73 & \textcolor{black}{0.52}  \\ 
                     & 8   & 4.19  &297.61&69.76& 367.37  &  \textcolor{black}{0.54}\\ 
                     & 10  & 3.80  &370.17&85.26& 455.43 &  \textcolor{black}{0.57 }\\ 
                     & 120 & 0.00  &0.00 & 930.10& 930.10 & \textcolor{black}{1.00 }    \\ \hline
\multirow{6}{*}{10} &0& 12.22&  0.00 &7.23& 7.23 &\textcolor{black}{0.00  }\\ 
                     & 2   & 8.55   &82.06&21.71&103.77  &  \textcolor{black}{0.29}\\ 
                     & 4   & 5.54   &105.52&36.18&141.71 & \textcolor{black}{0.53 } \\ 
                     & 6   & 5.10   &147.78&50.66&198.43  & \textcolor{black}{0.57} \\ 
                     & 8   & 4.52  &293.18&65.13 & 358.05   &\textcolor{black}{ 0.62}\\ 
                     & 10  & 4.25   &371.35&79.60& 450.95   & \textcolor{black}{0.64}\\ 
                     & 120 & 0.00  &0.00 & 868.39& 868.39   &\textcolor{black}{1.00}\\ \hline
\multicolumn{2}{l|}{Average} & 5.44 & 142.79 & 167.00 & 309.79&\textcolor{black}{0.51} \\ \hline
\end{tabular}
\end{small}
\end{table}

\textcolor{black}{The results show that for $K=6$ relocations in the planning horizon (which is reasonable in most applications) our approach allowed to achieve more than 50\% of the GAIN obtainable in an ``ideal'' situation in which a relocation is made possible at each period. A larger number of relocations provides only a negligible improvement.}

As far as the execution times are concerned, we observe that the first phase was always skipped for both $K=0$ and $K=120$. Indeed, in these cases the optimization  model (\ref{1c1})-(\ref{1c8}) had a unique solution, that is $\mathcal{T}_R=\mathcal{T}$ and $\mathcal{T}_R=\emptyset$, respectively. In all the other cases the heuristic solution $(\bar{\mathbf{O}}(\mathcal{T}_R),\bar{\mathbf{S}}(\mathcal{T}_R))$ was determined in $100.11$ and $309.79$ seconds on average for Tables \ref{Tab1} and \ref{Tab2}. For $|C|=50$, the majority of the execution time was spent during Phase I ($81.58$ seconds on the average). On the other hand, for $|C|=100$ the computing times for the two phases were comparable ($142.79$ versus $167.00$ seconds on the average).

%\section{Conclusions}\label{conclusions}
%This paper has studied a time-dependent $p$-center problem in which facilities are mobile units that can be relocated multiple times. We have proposed a multi-period formulation that accounts for the influence of traffic variability. Sufficient optimality conditions as well as other properties have been developed. The outcomes have been exploited to design a two-stage heuristic. Computational experiments have been carried out on instances based on the Paris (France) road graph. The results have indicated that the algorithm can consistently generate  solutions with a limited number of relocations that improve on the corresponding time-invariant optimal solutions.
%Future work will focus on adapting the ideas of this paper to other variants of the $p$-center problem, such as the \textit{capacitated} $p$-CP \citep{kramer2018mathematical}, the  \textit{weighted} $p$-CP \citep{ChenHandler1993}, the \textit{conditional} $p$-CP \citep{berman2008new}, the $\alpha$-\textit{neighbor} $p$-CP \citep{CHEN201336}. \textcolor{black}{Another future research direction we propose is to consider the extension of our results to the more complex multi-period location-routing problem \citep{ALBAREDASAMBOLA2012248}.}
\section{Conclusions}\label{Conclusions}
\textcolor{black}{This paper studied a time-dependent $p$-center problem in which facilities are mobile units that can be relocated multiple times. We  proposed a multi-period formulation that accounts for the influence of traffic variability. We assumed that the facilities can be relocated several times. In particular, we supposed that the service time from facilities to demand site are continuous piecewise linear function satisfying the FIFO property. The time variability of the worst case service time of each arc is modeled as constant stepwise function, whose \textit{pieces} are the maximum service time between two consecutive possible relocation time instants. We proposed a heuristic algorithm consisting of two phases. During the first phase an ILP problem is solved in order to determine relocations times. During the second phase, a \textit{classical} $p$-center problems is solved in order to determine the locations of the facilities between two consecutive relocation times.  We finally devised sufficient optimality conditions for an optimal solution. Computational experiments have been carried out on instances based on the Paris (France) road graph. The results  indicated that the algorithm can consistently generate  solutions with a limited number of relocations that improve on the corresponding time-invariant optimal solutions. Future work will focus on adapting the ideas introduced in this paper to other variant of the $p$-center problem, such as, among others, the \textit{capacitated} $p$-CP \citep{doi:10.1287/ijoc.2019.0889}, the  \textit{weighted} \textit{p}CP \citep{ChenHandler1993}, the \textit{conditional p}CP \citep{berman2008new}, the $\alpha$-\textit{neighbor p}CP \citep{CHEN201336}. Another future research direction we propose is to consider the extension of our results to the more complex multi-period location-routing problem \citep{ALBAREDASAMBOLA2012248}.}

\section*{References}

%\bibliographystyle{elsarticle-num-names}\biboptions{authoryear}
\bibliographystyle{plainnat}%\biboptions{authoryear}
\bibliography{biblio}


%\begin{thebibliography}{00}



%\end{thebibliography}
\end{document}